\documentclass[12pt]{amsart}
\usepackage{amsmath,times,epsfig,amssymb,amsbsy,amscd,amsfonts,amstext,color,bm}
\usepackage[margin=1in]{geometry}
\usepackage{enumerate}
\usepackage{placeins}
\usepackage{array}
\usepackage{tabulary}
\usepackage{multirow}
\usepackage{graphicx}
\usepackage{hyperref}
\usepackage{hhline}
\usepackage{soul}
\usepackage{multirow}
\usepackage[table, svgnames, dvipsnames]{xcolor}
\usepackage{makecell, cellspace, caption}
\usepackage{subcaption}
\usepackage[color,matrix,arrow]{xypic}
\usepackage{tikz}
        \usetikzlibrary{arrows.meta}
        \usetikzlibrary{arrows}
        \usetikzlibrary{quotes}
        \usetikzlibrary{graphs}
        \usetikzlibrary{positioning}
\usepackage{tikz-cd}
\hypersetup{
  colorlinks   = true, %Colours links instead of ugly boxes
  urlcolor     = blue, %Colour for external hyperlinks
  linkcolor    = blue, %Colour of internal links
  citecolor   = red %Colour of citations
}

\theoremstyle{plain}
\newtheorem{theorem}{Theorem}[section]
\newtheorem{corollary}[theorem]{Corollary}
\newtheorem{lemma}[theorem]{Lemma}

\newtheorem{proposition}[theorem]{Proposition}

\makeatletter
    
    \@addtoreset{equation}{section}
  \makeatother

\theoremstyle{definition}
\newtheorem{definition}[theorem]{Definition}
\newtheorem{example}[theorem]{Example}

\theoremstyle{remark}
\newtheorem{remark}[theorem]{Remark}

\newcommand{\A}{\mathcal{A}}

\newcommand{\C}{\mathbb{C}}
\newcommand{\F}{\mathbb{F}}

\newcommand{\scL }{\mathcal{L}}

\newcommand{\scJ}{\mathcal{J}}

\newcommand{\Q}{\mathbb{Q}}
\newcommand{\R}{\mathbb{R}}

\newcommand{\scR}{\mathcal{R}}
\newcommand{\scS}{\mathcal{S}}

\newcommand{\scP}{\mathcal{P}}

\newcommand{\Z}{\mathbb{Z}}

\newcommand{\h}{{\rm h}}

\newcommand{\asc}{{\rm asc}}
\newcommand{\dsc}{{\rm dsc}}
\newcommand{\I}{{\mathcal{I}}}

\newcommand{\M}{\mathcal{M}}

\newcommand{\lcm}{\operatorname{lcm}}

\newcommand{\quasi}{\operatorname{quasi}}

\newcolumntype{K}[1]{>{\centering\arraybackslash}p{#1}}

\begin{document}

\title[Eulerian polynomials for Weyl subarrangements]{Eulerian polynomials for subarrangements of Weyl arrangements}

\date{\today}

\begin{abstract}
\label{sec:intro}
Let $\mathcal{A}$ be a Weyl arrangement. We introduce and study the notion of $\mathcal{A}$-Eulerian polynomial producing an Eulerian-like polynomial for any subarrangement of $\mathcal{A}$. This polynomial together with shift operator describe how the characteristic quasi-polynomial of a new class of arrangements containing ideal subarrangements of $\mathcal{A}$ can be expressed in terms of the Ehrhart quasi-polynomial of the fundamental alcove. The method can also be extended to define two types of deformed Weyl subarrangements containing the families of the extended Shi, Catalan, Linial arrangements and to compute their characteristic quasi-polynomials. We obtain several known results in the literature as specializations, including the formula of the characteristic polynomial of $\mathcal{A}$ via Ehrhart theory due to Athanasiadis (1996), Blass-Sagan (1998), Suter (1998) and Kamiya-Takemura-Terao (2010); and the formula relating the number of coweight lattice points in the fundamental parallelepiped with the Lam-Postnikov  Eulerian polynomial due to the third author. 
 \end{abstract}

 \author{Ahmed Umer Ashraf}
\address{Ahmed Umer Ashraf, University of Western Ontario, Canada.}
\email{aashra9@uwo.ca}
\author{Tan Nhat Tran}
\address{Tan Nhat Tran, Department of Mathematics, Hokkaido University, Kita 10, Nishi 8, Kita-Ku, Sapporo 060-0810, Japan.}
\email{trannhattan@math.sci.hokudai.ac.jp}
\author{Masahiko Yoshinaga}
\address{Masahiko Yoshinaga, Department of Mathematics, Hokkaido University, Kita 10, Nishi 8, Kita-Ku, Sapporo 060-0810, Japan.}
\email{yoshinaga@math.sci.hokudai.ac.jp}

%\dedicatory{}


\date{\today}
\maketitle

%\tableofcontents

\section{Introduction}
 \textbf{Motivation.}
One of typical problems in enumerative combinatorics is to count the sizes of sets depending upon a positive integer $q$.
This often gives rise to polynomials. 
For instance, the \emph{chromatic polynomial} of an undirected graph, going back to Birkhoff and Whitney, encodes the number of ways of coloring the vertices with $q$ colors so that adjacent vertices get different colors.
However, it may happen that enumerating the cardinalities of sets leads to quasi-polynomials. 
Generally speaking, a quasi-polynomial is a refinement of polynomials, of which the coefficients may not come from a ring but instead are periodic functions with integral periods. 
Thus a quasi-polynomial is made of a bunch of polynomials, the \emph{constituents} of the quasi-polynomial.
One of the most classical examples in the theory is that the number of integral points in the $q$-fold dilation of a rational polytope agrees with a quasi-polynomial in $q$, broadly known as the \emph{Ehrhart quasi-polynomial}.

We are interested in the connection between the counting problems and the arrangement theory. 
A finite list (multiset) $\A$ of vectors in $\Z^\ell$ determines an arrangement $\A(\R)$ of hyperplanes in the vector space $\R^\ell$, an arrangement $\A(\mathbb{S}^1)$ of subtori in the torus $(\mathbb{S}^1)^\ell$, and especially an arrangement $\A(\Z_q)$ of subgroups in the finite abelian group $\Z_q^\ell$.
Enumerating the cardinality of the complement of  $\A(\Z_q)$ produces a quasi-polynomial, the  \emph{characteristic quasi-polynomial} $\chi_{\A}^{\mathrm{quasi}}(q)$ of $\A$ \cite{KTT08}.
This single quasi-polynomial encodes a number of combinatorial and topological information of several types of arrangements and has generated increasing interest recently (e.g., \cite{CW12, BM14,  Y18L, Y18W, TY19, T19}). 
Among the others, $\chi_{\A}^{\mathrm{quasi}}(q)$ has the first constituent identical with the \emph{characteristic polynomial} $\chi_{\A(\R)}(t)$ of $\A(\R)$ which justified its name  (e.g., \cite{A96, KTT08}), and the last constituent identical with the characteristic polynomial $\chi_{\A(\mathbb{S}^1)}(t)$ of $\A(\mathbb{S}^1)$ \cite{LTY, TY19}. 
One of the methods used in \cite{KTT08} for showing that $\chi_{\A}^{\mathrm{quasi}}(q)$ is indeed a quasi-polynomial is to express it as a sum of the Ehrhart quasi-polynomials of rational polytopes, or in the sense of \cite{BZ06}, as the Ehrhart quasi-polynomial of an  ``inside-out'' polytope. 
Such an expression is certainly interesting as it reveals the connection between two seemingly unrelated quasi-polynomials, one would hope for a more explicit expression if the list $\A$ was chosen to be a more special vector configuration.

 \textbf{Objective.} 
A particularly well-behaved class of the hyperplane arrangements is that of \emph{Weyl arrangements}. 
More precisely, if $\A=\Phi^+$ is a positive system of an irreducible root system $\Phi$, then $\A(\R)$ is called the Weyl arrangement of $\A$. 
It is proved that $\chi_{\Phi^+}^{\mathrm{quasi}}(q)$ is expressed in terms of the Ehrhart quasi-polynomial ${\rm L}_{A^\circ}(q)$ of the fundamental alcove $A^\circ$, the Weyl group, and the index of connection of  $\Phi$ (e.g., \cite{A96, BS98, Su98, KTT10}).
Thus we arrive at a natural and essential problem that for which subset $\Psi\subseteq \Phi^+$, $\chi_{\Psi}^{\mathrm{quasi}}(q)$ can be computed  by using the fundamental invariants of $\Phi$, and more importantly, by means of the Ehrhart quasi-polynomials.
Some partial results are known. 
If the root system $\Phi$  is of classical type and $\Psi$ is an \emph{ideal} of $\Phi^+$, then $\chi_{\Psi}^{\mathrm{quasi}}(q)$ can be computed from information of the signed graph associated with $\Psi$ \cite{T19}. 
A result due to \cite{Y18W} applied to any root system, asserts that $\chi_{\emptyset}^{\mathrm{quasi}}(q)$ (or simply $q^\ell$) can be written in terms of the Lam-Postnikov  \emph{Eulerian polynomial} \cite{LP18}, shift operator, and ${\rm L}_{A^\circ}(q)$. 

 \textbf{Results.} 
Inspired by the works of  \cite{LP18} and  \cite{Y18W}, we introduce the notion of \emph{$\mathcal{A}$-Eulerian polynomial} $E_\Psi(t)$ - an arrangement theoretical generalization of the classical Eulerian polynomial, and the notion of \emph{compatible subsets} of $\Phi^+$ w.r.t. the \emph{Worpitzky partition}. 
The first main result in our paper is a formula of $\chi_{\Psi}^{\mathrm{quasi}}(q)$ in terms of $E_\Psi(t)$, shift operator, and ${\rm L}_{A^\circ}(q)$ when  $\Psi\subseteq \Phi^+$ is compatible. 
 The formula specializes correctly to the two formulas in the extreme cases ($\Psi=\Phi^+$ and $\Psi=\emptyset$) mentioned above. 
In addition, we prove that the formula characterizes the compatibility.
The second main result in our paper is that the class of compatible subsets contains the ideals of the root system.
 Using the similar method, we further define two types of deformed Weyl subarrangements containing the families of the extended Shi, Catalan, Linial arrangements and compute their characteristic quasi-polynomials. 

 \textbf{Organization of the paper}. 
The remainder of the paper is organized as follows. 
In Section \ref{sec:Worpitzky}, we recall definitions and basic facts of irreducible root systems, their (affine) Weyl groups and the Worpitzky partition.
In Section \ref{sec:Characteristic-Ehrhart}, we recall the definitions of the characteristic and Ehrhart quasi-polynomials and specify the choices of lattices for these quasi-polynomials (Remarks \ref{rem:cha-lattice} and \ref{rem:E-lattice}). 
We also recall the formula between the quasi-polynomials in the extreme case $\Psi=\Phi^+$
(Theorem \ref{thm:quasi-Ehrhart}), and derive a more general formula (Proposition \ref{prop:bijection}).
In Section \ref{sec:generating}, we introduce the notion of $\mathcal{A}$-Eulerian polynomial (Definition \ref{def:Eulerian}), of which the main specialization is the Lam-Postnikov  Eulerian polynomial (Remark \ref{rem:empty-full1}). 
We then define the notion of compatible sets (Definition \ref{def:compatibleW}) which interpolates between the two extreme cases $\Psi=\Phi^+$ and $\Psi=\emptyset$.
We prove that the $\mathcal{A}$-Eulerian polynomial is an essential tool to compute $\chi_{\Psi}^{\mathrm{quasi}}(q)$ and its generating function for any compatible subset $\Psi\subseteq \Phi^+$ (Theorems \ref{thm:shift} and  \ref{thm:Eulerian}). 
We also prove that every ideal is compatible (Theorem  \ref{thm:ideal-com}).
In Section \ref{sec:deform},
we define two types of the deformed Weyl subarrangements (Definition \ref{def:types}), of which the main examples are the truncated affine Weyl and deleted Shi arrangements (Remark \ref{rem:deform}). 
Then we compute the characteristic quasi-polynomials of these deformed arrangements according to two special choices of intervals (Theorems \ref{thm:CQP-I} and \ref{thm:CQP-II}).

\section{Root systems and Worpitzky partition}
\label{sec:Worpitzky} 
Our standard reference for root systems and their Weyl groups is \cite{H90}.
Assume that $V=\R^\ell$ with the standard inner product $(\cdot,\cdot)$.
 Let $\Phi$ be an irreducible (crystallographic) root system in $V$ with 
the Coxeter number ${\rm h}$ and the Weyl group $W$. 
Fix a positive system $\Phi^+ \subseteq \Phi$ and let $\Delta := \{\alpha_1,\ldots,\alpha_\ell \}$ be the set of simple roots (base) of $\Phi$ associated with $\Phi^+$. 
For $\alpha = \sum_{i=1}^\ell d_i \alpha_i\in \Phi^+$, the \textit{height} of $\alpha$ is defined by $ {\rm ht}(\alpha) :=\sum_{i=1}^\ell  d_i$. 
Define the partial order $\ge$ on $\Phi^+$ such that $\beta_1 \ge \beta_2$ if and only if $\beta_1-\beta_2 = \sum_{i=1}^\ell n_i\alpha_i$ with all $n_i \in \Z_{\ge 0}$. 
The highest root (w.r.t. $\ge$), denoted by $\tilde{\alpha} \in \Phi^+$, can be written uniquely as a linear combination of the simple roots $\tilde{\alpha}= \sum_{i=1}^\ell c_i\alpha_i$  ($c_i \in  \Z_{>0}$). 
Set $\alpha_0 := -\tilde{\alpha}$, $c_0 := 1$, and $\widetilde{\Delta}:=\Delta \cup \{\alpha_0\}$. 
Then we have the linear relation
\begin{equation*}
\label{eq:relation}
c_0\alpha_0 +c_1\alpha_1+\cdots+c_\ell\alpha_\ell  =0.
\end{equation*}
The coefficients $c_i$ are important in our study and will appear frequently throughout the paper.

For $\alpha \in V\setminus\{0\}$, denote $\alpha^\vee:=\frac{2 \alpha}{(\alpha,\alpha)}$.
The \emph{root lattice} $Q(\Phi)$, \emph{coroot lattice} $Q(\Phi^\vee)$, \emph{weight lattice} $Z(\Phi)$, and  \emph{coweight lattice} $Z(\Phi^\vee)$ are defined as follows:
\begin{align*}
Q(\Phi) &:= \bigoplus_{i=1}^\ell \Z\alpha_i, \\
Q(\Phi^\vee) &:= \bigoplus_{i=1}^\ell \Z\alpha_i^\vee,\\
Z(\Phi) &:= \{x\in V \mid(\alpha_i^\vee ,x)\in \Z\,(1 \le i \le \ell)\}, \\
Z(\Phi^\vee) &:= \{x\in V \mid(\alpha_i ,x)\in \Z\,(1 \le i \le \ell)\}.
\end{align*}
Then $Q(\Phi)$ is a subgroup of $Z(\Phi)$ of finite index $f$, and similarly $Q(\Phi^\vee)$ is a subgroup of $Z(\Phi^\vee)$ of the index $f$. 
The number $f$ is called the \emph{index of connection}.  
Let $\{\varpi^\vee_1, \ldots ,\varpi^\vee_\ell\} \subseteq Z(\Phi^\vee)$ be the dual basis of the base $\Delta$, namely, $(\alpha_i ,\varpi^\vee_j) = \delta_{ij}$. 
Then $Z(\Phi^\vee)= \bigoplus_{i=1}^\ell \Z\varpi_i^\vee$ and $c_i = (\varpi^\vee_i , \tilde{\alpha} )$.

For $m \in \Z$ and $\alpha \in  \Phi$, the \emph{affine hyperplane} $H_{\alpha,m}$ is defined by
$$H_{\alpha,m} :=\{x\in V \mid(\alpha,x)=m\}.$$
A connected component of $V \setminus \bigcup_{ \alpha\in \Phi^+,m \in \Z} H_{\alpha,m}$ is called an \emph{alcove}. 
The \emph{fundamental alcove} $A^\circ$ is defined by 
$$A^\circ := \left\{x \in V \middle|
\begin{array}{c}
      (\alpha_i ,x)> 0\,(1 \le i \le \ell),  \\
      (\alpha_0,x)>-1
    \end{array}
\right\}.
$$
The closure $\overline{A^\circ} =\{x\in V \mid(\alpha_i ,x)\ge 0\,(1 \le i \le \ell), (\alpha_0,x)\ge -1\}$ is a simplex, which is the convex hull of $\{0,\varpi^\vee_1/c_1 ,\ldots,\varpi^\vee_\ell/c_\ell\}$. 
The supporting hyperplanes of the facets of $\overline{A^\circ}$ are $H_{\alpha_1,0}, \ldots , H_{\alpha_\ell ,0}, H_{\alpha_0, -1}$. 
The affine Weyl group $W_{\rm aff} := W \ltimes Q(\Phi^\vee)$ acts simply transitively on the set of alcoves and admits $\overline{A^\circ}$ as a fundamental domain for its action on $V$.

The fundamental domain $P^\diamondsuit$ of the coweight lattice $Z(\Phi^\vee)$, called the \emph{fundamental parallelepiped}, is defined by
\begin{align*}
P^\diamondsuit  &:= \sum_{i=1}^\ell (0,1] \varpi^\vee_i\\
 &=  \{x\in V \mid 0<(\alpha_i ,x)\le 1\,(1 \le i \le \ell)\}.
\end{align*}
Let $N:=\frac{\#W}{f}$, and denote $[N]:=\{1,2,\ldots,N\}$. 
Then the  cardinality of the set $\Xi$ of all alcoves contained in $P^\diamondsuit$ equals $N$ (see, e.g., \cite[Theorem 4.9]{H90}). 
Let us write 
$$\Xi=\{A_i^\circ \subseteq P^\diamondsuit \mid i \in [N]\},$$  
where each alcove $A_i^\circ$ is written uniquely as
$$A_i^\circ= \left\{x \in V \middle|
\begin{array}{c}
      (\alpha ,x)> k_\alpha \,(\alpha \in I),  \\
     (\beta ,x)< k_\beta \,(\beta \in J)
    \end{array}
\right\},
$$
where $k_\alpha=k_\alpha(A_i^\circ) \in \Z_{\ge0}$, $k_\beta=k_\beta(A_i^\circ)  \in \Z_{>0}$ and the sets $I,J \subseteq \Phi^+$ with $\#(I \sqcup J)=\ell+1$ indicate the constraints on $x\in V$ according to the inequality symbols $>$, $<$, respectively.

 \begin{definition}\label{def:partial}   
For each $A_i^\circ \in\Xi$, the \emph{partial closure} $A_i^\diamondsuit$ of $A_i^\circ$ is defined by
 $$
\label{eq:A-diamond}
A_i^\diamondsuit
:= \left\{x \in V \middle|
\begin{array}{c}
      (\alpha ,x)> k_\alpha \,(\alpha \in I),  \\
     (\beta ,x)\le k_\beta \,(\beta \in J)
    \end{array}
\right\}.
 $$

\end{definition}

\begin{theorem}[Worpitzky partition]
\label{thm:Worpitzky}
$$P^\diamondsuit=\bigsqcup_{i \in [N]}A_i^\diamondsuit.$$
\end{theorem}
 \begin{proof} 
See, e.g., \cite[Proposition 2.5]{Y18W}, \cite[Exercise 4.3]{H90}.
\end{proof}

\section{Characteristic and Ehrhart quasi-polynomials}
\label{sec:Characteristic-Ehrhart}
A function $g: \Z \to \C$ is called a \emph{quasi-polynomial} if there exist $\rho\in\Z_{>0}$ and polynomials $f^k(t)\in\Z[t]$ ($1 \le k \le \rho$) such that for any $q\in\Z_{>0}$ with  $q\equiv k\bmod \rho$, 
\begin{equation*}
g(q) =f^k(q).
\end{equation*}
The number $\rho$ is called a \emph{period} and the polynomial $f^k(t)$ is called the \emph{$k$-constituent} of the quasi-polynomial $g$. 

Let $\Gamma:=\bigoplus_{i=1}^\ell \Z\beta_i \simeq \Z^\ell$ be a lattice with a basis $\{\beta_1,\ldots, \beta_\ell\}$.
Let $\scL$ be a finite list (multiset) of elements in $\Gamma$. 
Let $q\in\Z_{>0}$ and denote $\Z_q:=\Z/q\Z$.
For $\alpha=\sum_{i=1}^\ell  a_i\beta_i \in \scL$ and $m_\alpha \in \Z$, define a subset $H_{\alpha, m_\alpha,\Z_q}$ of $\Z_q^\ell$ by 
$$H_{\alpha,m_\alpha, \Z_q}:=\Big\{\textbf{z} \in  \Z_q^\ell  \mid \sum_{i=1}^\ell \overline{a_i}z_i\equiv\overline{m_\alpha }\Big\}.$$
Given a vector $m=(m_\alpha)_{ \alpha  \in \scL}\in\Z^\scL$,  we can define the \emph{$\Z_q$-plexification (or $q$-reduced)} arrangement of $(\scL,m)$ by
$$(\scL,m)(\Bbb \Z_q):=\{H_{\alpha,m_\alpha, \Z_q}\mid \alpha  \in \scL\}.$$
The \emph{complement} of  $(\scL,m)(\Bbb \Z_q)$ is defined by
$$\M((\scL,m);\Gamma,\Z_q):= \Z_q^\ell \setminus\bigcup_{\alpha\in\scL}H_{\alpha,m_\alpha, \Z_q}.$$
For a sublist $\scS\subseteq \scL$, write $d_{\scS, n_{\scS}}$ for the \emph{largest invariant factor} of the torsion subgroup of the finitely generated abelian group $\Gamma/\langle\scS\rangle$ (see, e.g., \cite[Definition 3.8]{N12}).
The \emph{LCM-period} $\rho_{\scL}$ of $\scL$ is defined by 
\begin{equation*}
\label{eq:LCM-period}
\rho_\scL:=\lcm(d_{\scS, n_{\scS}}\mid \scS\subseteq \scL). 
\end{equation*} 
\begin{theorem}
\label{thm:KTT}
$\#\M((\scL,m);\Gamma,\Z_q)$ is a monic quasi-polynomial in $q$ for which $\rho_{\scL}$ is a period. 
The quasi-polynomial is called the \emph{characteristic quasi-polynomial} of $(\scL,m)$ w.r.t. the lattice $\Gamma$, and denoted by $\chi^{\quasi}_{(\scL,m)}(q)$. 
\end{theorem}
 \begin{proof} 
See  \cite[Theorem 2.4]{KTT08} and  \cite[Theorem 3.1]{KTT11}.
\end{proof}

We can also define the \emph{$\R$-plexification} (in fact, the real hyperplane arrangement) of $(\scL,m)$ as follows:
$$(\scL,m)(\R):=\{H_{\alpha,m_\alpha, \Bbb \R} \mid \alpha  \in \scL\},$$ 
where $H_{\alpha,m_\alpha, \R}:=\left\{\textbf{x} \in  \R^\ell  \mid  \sum_{i=1}^\ell a_ix_i=m_\alpha\right\}$. 
For a real hyperplane arrangement $\A$, denote by $\chi_\A( t)$ the characteristic polynomial (e.g., \cite[Definition 2.52]{OT92}) of $\A$.
\begin{theorem}
\label{thm:KTT11}
The first constituent of $\chi^{\quasi}_{(\scL,m)}(q)$ coincides with the characteristic polynomial of $(\scL,m)(\R)$, i.e.,
$$f^1_{(\scL,m)}(t)=\chi_{(\scL,m)(\R)}(t).$$
\end{theorem}
 \begin{proof} 
See, e.g., \cite[Theorem 2.5]{KTT08} and  \cite[Remark 3.3]{KTT11}.
\end{proof}

\noindent
\emph{Convention}: 
When $m=(0)$ the zero vector, we simply write $\scL$, $H_{\alpha, \Z_q}$,  $H_{\alpha, \R}$ instead of $(\scL,m)$, $H_{\alpha,0, \Z_q}$,  $H_{\alpha, 0,\R}$, respectively.

\begin{remark}
\label{rem:G-}
$\scL(\Bbb \Z_q)$ and $\scL(\R)$ are examples of a more abstract concept, the \emph{$G$-plexifications} ($G$ is an abelian group) introduced in \cite{LTY} by viewing $G=\Z_q, \R$, respectively (see also \cite{TY19} for more information on their combinatorial properties).
In addition, the characteristic quasi-polynomial can be generalized to \emph{chromatic quasi-polynomial} by replacing the lattice $\Gamma$ by a finitely generated abelian group, see, e.g., \cite{Tan18} for more details. 
\end{remark}

\begin{remark}
\label{rem:cha-lattice}
Throughout the paper, for every $\Psi\subseteq\Phi^+(\subseteq Q(\Phi))$, the characteristic quasi-polynomial $\chi^{\quasi}_{\Psi}(q)$ is always defined w.r.t. the root lattice $\Gamma=Q(\Phi)$.
\end{remark}

For each $\Psi\subseteq\Phi^+$, define the \emph{Weyl arrangement} of $\Psi$ by $\A_{\Psi}:= \{H_\alpha \mid \alpha\in\Psi\}$, where $H_\alpha=\{x\in V \mid (\alpha,x)=0\}$ is the hyperplane orthogonal to $\alpha$. 
It is not hard to see that $H_\alpha\simeq H_{\alpha, \R}$ (as vector spaces), thus we can view $\A_{\Psi}$ as the $\R$-plexification of $\Psi$, i.e., $\A_{\Psi}=\Psi(\R)$. 
In standard terminology, $\A_{\Phi^+}$ is known with the name \emph{Weyl (or Coxeter) arrangement}, and clearly $\A_{\Psi}$ is a  (Weyl) subarrangement of $\A_{\Phi^+}$.

\begin{remark}
\label{rem:another}
Sometimes, when we say ``the" characteristic quasi-polynomial of $(\scL,m)(\Bbb \Z_q)$ or of $(\scL,m)(\R)$, we mean the characteristic quasi-polynomial $\chi^{\quasi}_{(\scL,m)}(q)$ of $(\scL,m)$. 
For example, the characteristic quasi-polynomial $\chi^{\quasi}_{\A_{\Psi}}(q)$ of $\A_{\Psi}$  is referred to as $\chi^{\quasi}_{\Psi}(q)$. 
We will use this term later in Section \ref{sec:deform} when we deal with deformed Weyl arrangements of $\Psi$.
\end{remark}

Let $\Gamma$ be a lattice. 
For a polytope $\scP$ with vertices in the rational vector space generated by $\Gamma$,  the \emph{Ehrhart quasi-polynomial} ${\rm L}_{\scP}(q)$ of $\scP$ w.r.t. $\Gamma$ is defined by 
$${\rm L}_{\scP}(q):=\#(q\scP \cap \Gamma).$$
We denote by $\scP^\circ$ the relative interior of $\scP$. Similarly, we can define
$${\rm L}_{\scP^\circ}(q):=\#(q\scP^\circ \cap \Gamma).$$
For $q > 0$, the following reciprocity law holds:
$${\rm L}_{\scP}(-q) = (-1)^{\dim \scP}{\rm L}_{\scP^\circ}(q).$$

\begin{remark}
\label{rem:E-lattice}
Throughout the paper, the Ehrhart quasi-polynomials ${\rm L}_{\overline{A^\circ}}(q)$, ${\rm L}_{A^\circ}(q)$ are defined w.r.t. the coweight lattice $\Gamma=Z(\Phi^\vee)$.
\end{remark}

Let  $F_0 :=  H_{\alpha_0,-q}, \, F_i := H_{\alpha_i,0}\,\,(1\le i \le \ell)$ denote the supporting hyperplanes of the facets of $q \overline{A^\circ}$. Then the number of coweight lattice points in $q \overline{A^\circ}$ after removing some facets can be computed by ${\rm L}_{\overline{A^\circ}}$ with the scale factor of dilation being reduced. 
\begin{proposition}
\label{prop:remove}   
Let $\{i_1, \ldots, i_k\} \subseteq \{0,1, \ldots, \ell\}$. Suppose that $q > c_{i_1} + \cdots + c_{i_k}$. 
Then
$$\#(q \overline{A^\circ} \cap Z(\Phi^\vee)\setminus (F_{i_1} \cup \cdots \cup F_{i_k}) ) = {\rm L}_{\overline{A^\circ}}(q-(c_{i_1} + \cdots + c_{i_k})).$$
In particular, for $q\in\Z$, 
$${\rm L}_{A^\circ}(q)={\rm L}_{\overline{A^\circ}}(q-\h).$$
\end{proposition} 
\begin{proof} 
See \cite[Corollaries 3.4 and 3.5]{Y18W}.
\end{proof}

In general, it is not easy to find explicit formulas which involve both characteristic and Ehrhart quasi-polynomials.
With regards to root systems, there is an interesting relation between these quasi-polynomials.
\begin{theorem} 
\label{thm:quasi-Ehrhart}
$$\chi^{\quasi}_{\Phi^+}(q) =\frac{\#W}{f}{\rm L}_{A^\circ}(q).$$
\end{theorem}
\begin{proof} 
See, e.g., \cite{KTT10}, \cite[Proposition 3.7]{Y18W}.
\end{proof}

\begin{theorem}
\label{thm:KTTproceed}
 The minimum period of $\chi^{\quasi}_{\Phi^+}(q)$ is equal to $\lcm(c_1,\ldots, c_\ell)$. 
 Furthermore, by a case-by-case argument, it can be checked that the minimum period coincides with the LCM-period $\rho_{\Phi^+}$.
\end{theorem}
 \begin{proof} 
See \cite[Corollary 3.2 and Remark 3.3]{KTT10}.
\end{proof}

\begin{corollary}
\label{cor:>0}  
 $\chi^{\quasi}_{\Phi^+}(q)>0$ (equivalently, ${\rm L}_{A^\circ}(q)>0$) if and only if $q \ge \h$.
 \end{corollary}
\begin{proof}
See, e.g., \cite[Corollary 3.4]{KTT10}.
\end{proof}

We can extend Theorem \ref{thm:quasi-Ehrhart} to a formula of $\chi^{\quasi}_{\Psi}(q)$ for any $\Psi\subseteq\Phi^+$ in terms of lattice point counting functions. 
In the proposition below, we view $\Psi$ as a list with possible repetitions of elements. 
Along the proof, we will see why the choice of lattices for the characteristic and Ehrhart quasi-polynomials is important.
\begin{proposition}
\label{prop:bijection}   
Let $m=(m_\alpha)_{ \alpha  \in \Psi}$ be a vector in $\Z^\Psi$.
Set 
\begin{align*}
X_{(\Psi,m)}(q) & := qP^\diamondsuit  \cap Z(\Phi^\vee) \setminus  \bigcup_{\alpha\in \Psi,k \in \Z}H_{\alpha, kq+m_\alpha}, \\
Y_{(\Psi,m)}(q) & := \{\overline{x} \in Z(\Phi^\vee)/qZ(\Phi^\vee) \mid (\alpha, x) \not\equiv m_\alpha\bmod q, \forall \alpha \in \Psi\}.
\end{align*}
We have bijections between sets
$$X_{(\Psi,m)}(q) \simeq Y_{(\Psi,m)}(q)  \simeq \M((\Psi,m); \Z^\ell, \Z_q).$$
As a result, 
$$\chi^{\quasi}_{(\Psi,m)}(q)=\#X_{(\Psi,m)}(q)= \# Y_{(\Psi,m)}(q).$$
\end{proposition} 
\begin{proof} 
The bijection $X_{(\Phi^+,m)}(q) \simeq Y_{(\Phi^+,m)}(q)$ is proved in  \cite[\S3.3]{Y18W}. 
We can use exactly the same argument applied to every subset $\Psi$.
The proof of $Y_{\Psi}(q)  \simeq \M(\Psi; \Z^\ell, \Z_q)$ for an arbitrary $\Psi \subseteq \Phi^+$ runs as follows:
\begin{align*}
Y_{(\Psi,m)}(q) 
 & = \{ \overline{x} = \sum_{i=1}^\ell  \overline{z_i} \varpi^\vee_i \mid (\alpha,x)\not\equiv m_\alpha\bmod q, \forall \alpha \in \Psi\}\\
 & = \{ \overline{x} = \sum_{i=1}^\ell \overline{z_i}\varpi^\vee_i \mid ( \sum_{i=1}^\ell  S_{ij}\alpha_i, \sum_{i=1}^\ell  z_i \varpi^\vee_i)\not\equiv m_\alpha\bmod q, \,(1 \le j \le \#\Psi) \}\\
 & \simeq \{ \textbf{z} = (\overline{z_1}, \ldots, \overline{z_\ell}) \in  \Z_q^\ell  \mid \sum_{i=1}^\ell  z_i S_{ij} \not\equiv m_\alpha\bmod q,  \,(1 \le j \le \#\Psi)\}\\
& =\M((\Psi,m); \Z^\ell, \Z_q).
\end{align*}
\end{proof}

\begin{remark}
\label{rem:partial-KTT}
The bijection $X_{\Phi^+}(q) \simeq \M( \Phi^+; \Z^\ell, \Z_q)$ appeared (without proof) in  \cite[Proof of Theorem 3.1]{KTT10}. 
Theorem \ref{thm:quasi-Ehrhart} is a special case of Proposition \ref{prop:bijection} because  $\chi^{\quasi}_{\Phi^+}(q)=\#X_{\Phi^+}(q)=\frac{\#W}{f}{\rm L}_{A^\circ}(q)$  \cite[\S3.3]{Y18W}. 
\end{remark}
 
\section{Eulerian polynomials for Weyl subarrangements}
\label{sec:generating} 
Let $\Psi \subseteq \Phi^+$, and set $\Psi^c:=\Phi^+ \setminus \Psi$. 

 \begin{definition}\label{def:dsc}   
The \emph{descent $\dsc_\Psi$ w.r.t. $\Psi$}  is the function $\dsc_\Psi: W \to \Z_{\ge0}$ defined by
$$
\dsc_\Psi(w) := \sum_{0 \le i \le \ell,\, w(\alpha_i)\in -\Psi^c}c_i.
$$
\end{definition}

 \begin{definition}\label{def:Eulerian}   
The \emph{(arrangement theoretical Eulerian or) $\A$-Eulerian polynomial} of $\Psi$ (or of the Weyl subarrangement $\A_{\Psi}$) is defined by
$$E_\Psi(t):= \frac1f \sum_{w\in W}t^{\h-\dsc_\Psi(w)}.$$
\end{definition}

\begin{remark}
\label{rem:empty-full1}
 \begin{enumerate}[(a)]
\item  
If $\Psi = \Phi^+$, then $\dsc_{\Phi^+}(w)=0$ for all $w\in W$, and $E_{\Phi^+}(t)= \frac{\#W}{f}t^\h$. 
\item  If $\Psi =\emptyset$, then $\dsc_\emptyset=\dsc=\rm{cdes}$, the \emph{circular descent statistic}, e.g., \cite[Definition 6.2]{LP18}, \cite[Definition 4.1]{Y18W}. 
Then $E_{\emptyset}(t)=R_\Phi(t)$, the \emph{generalized Eulerian polynomial}, e.g.,  \cite[Definition 4.4]{Y18W}. 
In this paper, we often call it the Lam-Postnikov Eulerian polynomial.
Note that if $\Phi$ is of type $A_\ell$, then $R_\Phi(t)=A_\ell(t)$,  the \emph{classical $\ell$-th Eulerian polynomial} \cite[Theorem 10.1]{LP18}. 
It should also be noted that $\dsc_\emptyset$ coincides with the notion of \emph{affine descents} by Dilks-Petersen-Stembridge \cite[\S2.5]{DPS09} only in type $A$ case.
\end{enumerate}
\end{remark}

\begin{lemma}
\label{lem:range}
For all $w\in W$, $0 \le \dsc_\Psi(w)< \h$. 
In particular, $E_\Psi(0)=0$.
 \end{lemma} 
\begin{proof} 
If $\Psi = \Phi^+$, then the statements are clear  by Remark \ref{rem:empty-full1}(a).
Assume that $\Psi \ne \Phi^+$. 
If $w(\alpha_i)\notin -\Psi^c$ for some $1 \le i \le \ell$, then $\dsc_\Psi(w)< \h$.
Otherwise, we have $w(\alpha_0)=-\sum_{i=1}^\ell c_iw(\alpha_i) \in \Phi^+$.
Thus $w(\alpha_0)\notin -\Psi^c$, and hence $\dsc_\Psi(w)< \h$.
\end{proof}

\begin{lemma}
\label{lem:extend}
\begin{enumerate}[(i)]
\item Let $w \in W$. Suppose that $w$ induces a permutation on $ \widetilde{\Delta} = \{ \alpha_0, \alpha_1, \ldots, \alpha_\ell\}$. If $w(\alpha_i)=\alpha_{p_i}$, then $c_i =c_{p_i}$.
\item  Let  $w_1, w_2 \in W$. If there exists $\gamma \in V$ such that $w_1( A^\circ)=w_2( A^\circ) + \gamma$, then $\dsc_\Psi(w_1)=\dsc_\Psi(w_2)$.
\end{enumerate}
\end{lemma} 
\begin{proof} 
(i) is exactly \cite[Lemma 4.3(1)]{Y18W}. (ii) is similar to  \cite[Lemma 4.3(2)]{Y18W}:
since the supporting hyperplanes coincide, $w_1(\alpha_i)=w_2(\alpha_{p_i})$ for $0 \le i\le \ell$. By Definition \ref{def:dsc} and (i), we have
$$
\dsc_\Psi(w_1) = \sum_{0 \le i \le \ell,\, w_1(\alpha_i)\in -\Psi^c}c_i
= \sum_{0 \le p_i \le \ell,\, w_2(\alpha_{p_i})\in -\Psi^c}c_{p_i}=\dsc_\Psi(w_2).
$$
\end{proof}

Let $A'$ be an arbitrary alcove. 
We can write $A'=w( A^\circ) + \gamma$ for some $w\in W$ and $\gamma \in Q(\Phi^\vee)$.
Recall the the Worpitzky partition from Theorem \ref{thm:Worpitzky} that $P^\diamondsuit=\bigsqcup_{i \in [N]}A_i^\diamondsuit$.
By Lemma \ref{lem:extend}, we can extend $\dsc_\Psi$ to a function on the set of all alcoves (in particular, on the set $\Xi$ of alcoves $A_i^\circ$ contained in $P^\diamondsuit$) as follows:

 \begin{definition}\label{def:asc-ext}   
$$\dsc_\Psi(A'):=\dsc_\Psi(w).$$
\end{definition}

\begin{theorem}\label{thm:equivalent}   
$$E_\Psi(t)= \sum_{i \in [N]} t^{\h-\dsc_\Psi(A_i^\circ)}.$$
\end{theorem} 
\begin{proof} 
Similar to  \cite[Theorem 4.7]{Y18W}: there are exactly $f$ elements of the group $\widehat{W_{\rm aff}} := W \ltimes Z(\Phi^\vee)$ that map $A^\circ$ to a given alcove in $P^\diamondsuit$.
\end{proof}

In what follows, we shall sometimes abuse terminology and call a face of $\overline{A_i^\circ}$   that is contained in the partial closure $A_i^\diamondsuit$, a face of $A_i^\diamondsuit$.
 \begin{definition}
 \label{def:compatibleW}   
A subset $\Psi \subseteq \Phi^+$ is said to be \emph{compatible} (with the Worpitzky partition) if for each $A_i^\diamondsuit \subseteq P^\diamondsuit$ the following condition holds:  
$A_i^\diamondsuit \cap H_{\alpha, m_\alpha}$ for $\alpha\in \Psi,m_\alpha \in \Z$ is either empty, or contained in $A_i^\diamondsuit \cap H_{\beta, m_\beta}$ for $\beta\in \Psi,m_\beta \in \Z$ so that   $A_i^\diamondsuit \cap H_{\beta, m_\beta}$ is a facet of $A_i^\diamondsuit$.
\end{definition}

\noindent 
In other words, $\Psi$ is compatible if and only if the intersection of each $A_i^\diamondsuit$ and $\bigcup_{\mu\in \Psi,k \in \Z}H_{\mu,k}$ consisting of $a$ facets $K_1,\ldots, K_a$, and $b$ faces $G_1,\ldots, G_b$ of dimension $\le \ell-2$ of $A_i^\diamondsuit$, either is empty, or satisfies the condition that for every non-facet $G_j$ there exists a facet $K_s$ such that $G_j \subseteq K_s$. 
Intuitively, every ``non-empty face intersection" of $A_i^\diamondsuit$ and the affine hyperplanes w.r.t. roots in $\Psi$ can be lifted to a ``facet intersection".

\begin{example}
 \label{ex:compatibleW}
 \begin{enumerate}[(a)]
\item 
Clearly, $\emptyset$ and $\Phi^+$ are compatible. 
Let $\delta \in \Phi^+$, then $\Psi=\Phi^+ \setminus\{\delta\}$ is compatible. 
It is because if $A_i^\diamondsuit \cap H_{\alpha, m_\alpha}$ is a non-empty face but not a facet of $A_i^\diamondsuit$ for $i \in [N]$, $\alpha\in \Psi,m_\alpha \in \Z$, then $A_i^\diamondsuit \cap H_{\alpha, m_\alpha}= \cap_{j=1}^n H_{\beta_j, m_{\beta_j}}\cap  A_i^\diamondsuit$, where $H_{\beta_j, m_{\beta_j}}$ ($\beta_j \in \Phi^+$) are the supporting hyperplanes of the facets of $A_i^\diamondsuit$ for all $1 \le j \le n$ with $n \ge 2$. 
Thus, there exists $1 \le k \le n$ such that $\beta_k \ne \delta$, i.e., $\beta_k \in \Psi$.
\item 
If $\Psi_1=\{\alpha\}$ with $\alpha = \sum_{i=1}^\ell d_i\alpha_i \in \Phi^+ \setminus\Delta$, then $\Psi_1$ is not compatible because $H_{\alpha, {\rm ht}(\alpha)}\cap P^\diamondsuit = \left\{\sum_{i=1}^\ell p_i\varpi^\vee_i  \mid \sum_{i=1}^\ell p_id_i = {\rm ht}(\alpha), p_i \in  (0,1]\right\} = \left\{ \sum_{i=1}^\ell \varpi^\vee_i\right\}$. 
There exists a non-compatible set consisting of more than one elements. 
For example, let $\Phi=G_2$ as in Figure \ref{fig:G2}. 
Then it is easily seen from Figure \ref{fig:q=9} that $\Psi' = \{\alpha_1, \alpha_1+\alpha_2, 3\alpha_1+2\alpha_2\}$ is not compatible (regardless of the value of $q$) because there exist partially closed alcoves (in green) such that the affine hyperplanes (in pink) w.r.t. roots in $\Psi'$ intersect them only at their vertices.
\end{enumerate}
\end{example}

\begin{figure}[!ht]
\centering
    \includegraphics[width=9cm,height=9cm]{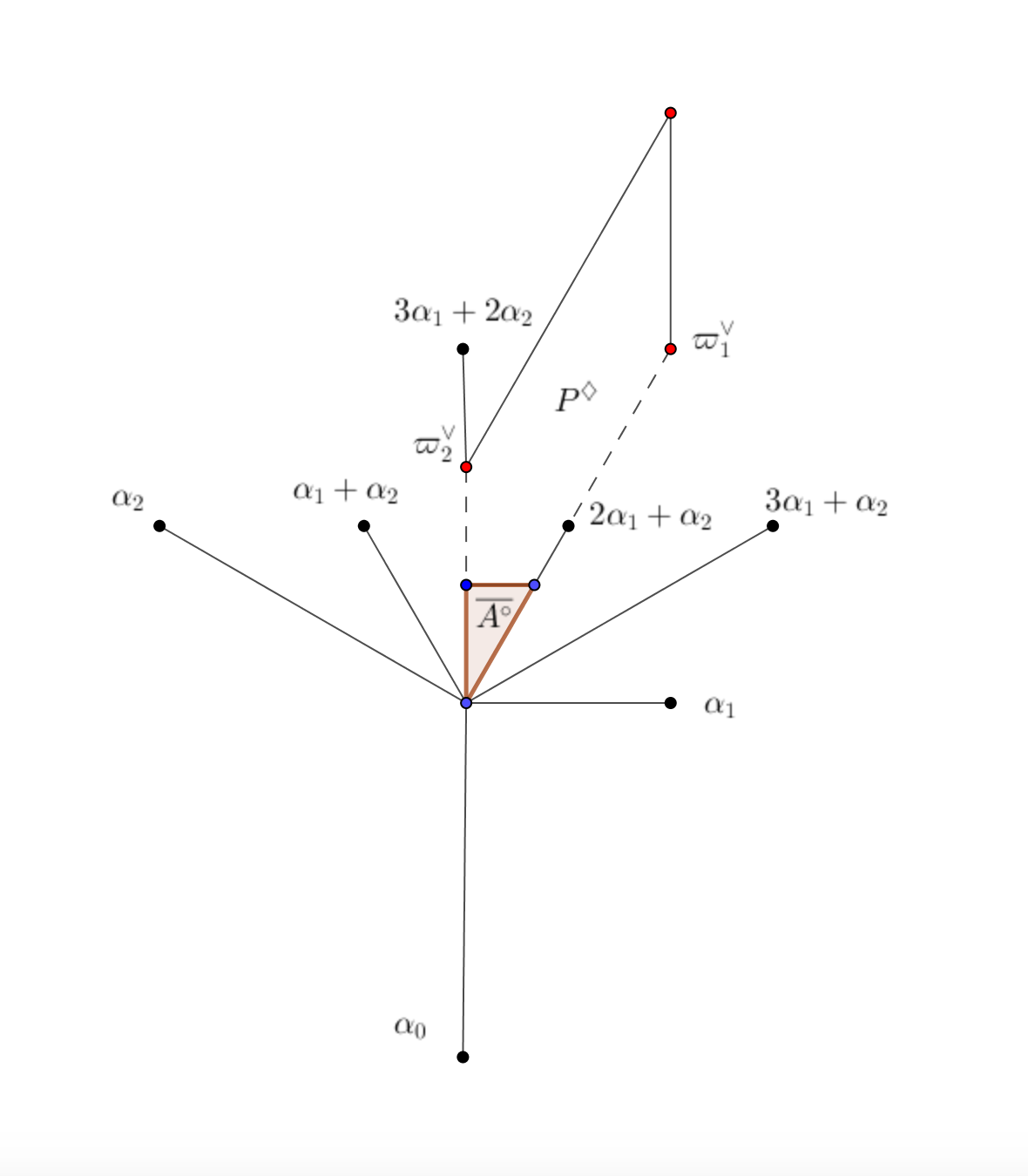}
    \caption{Root system of type $G_2$.}
    \label{fig:G2}
\end{figure}

\begin{figure}[!ht]
\centering
    \includegraphics[width=10cm,height=12cm]{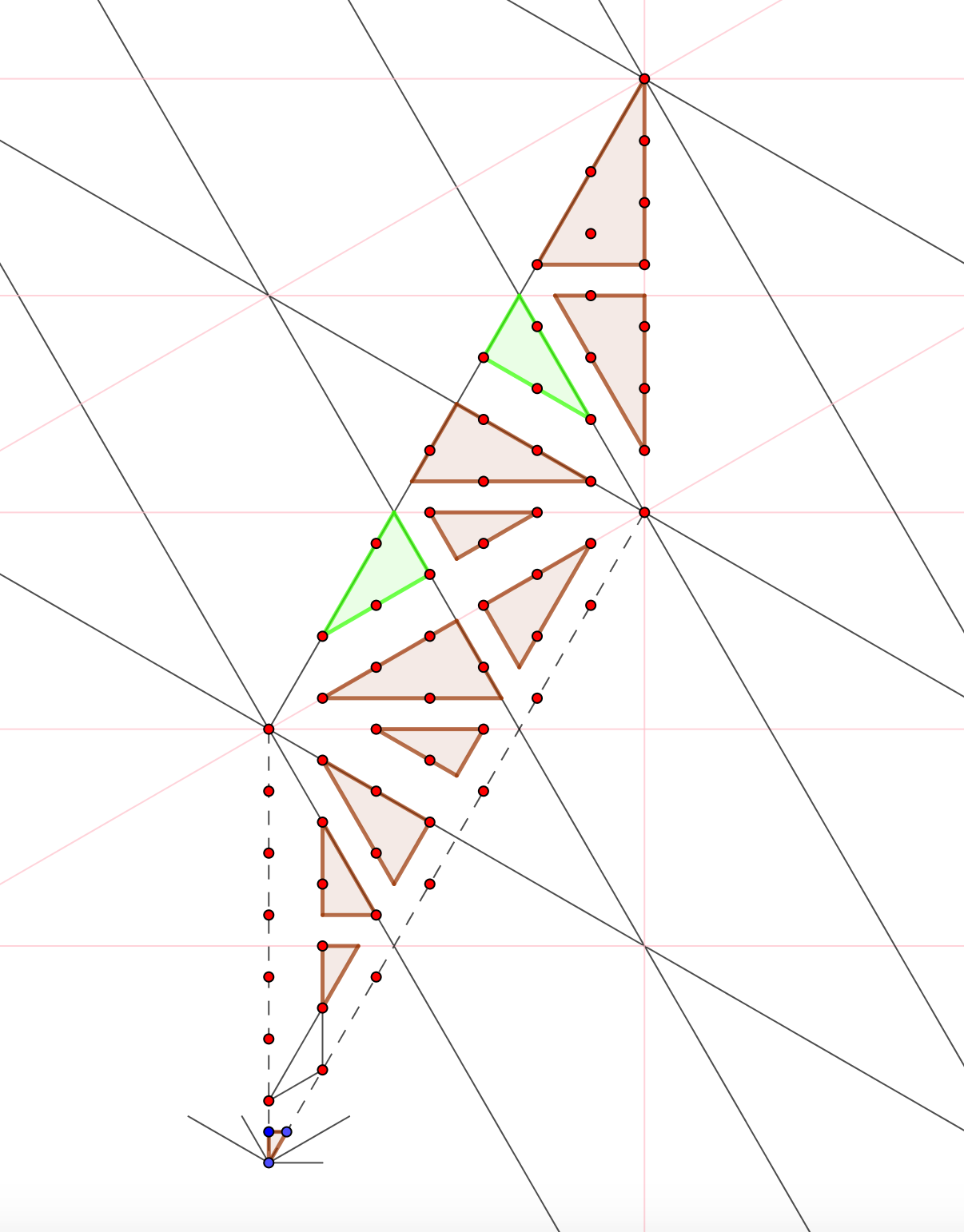}
    \caption{The Worpitzky partition of $qP^\diamondsuit\cap Z(\Phi^\vee)$ in a type $G_2$ root system ($q=7$).}
    \label{fig:q=9}
\end{figure}

 \begin{definition}
 \label{def:shift-operator}   
Let $f: \mathbb{Z} \to \mathbb{C}$ be a function and let $P(S)=􏰇\sum_{k=0}^na_kS^k$ be a polynomial in $S$. 
The \emph{shift operator} via $P(S)$ acting on $f$ is defined by
$$(P(S)f)(t):=\sum_{k=0}^n a_kf(t-k).$$
\end{definition}

 Now let us prove the first main result in the paper. 
\begin{theorem}\label{thm:shift}   
A subset $\Psi \subseteq \Phi^+$ is compatible if and only if for every positive integer $q$,
$$\chi^{\quasi}_{\Psi}(q) = (E_\Psi(S){\rm L}_{\overline{A^\circ}})(q).$$
\end{theorem} 
\begin{proof} 
The proof of ``$\Rightarrow$" is similar in spirit to \cite[Proof of Theorem 4.8]{Y18W}, 
but requires a more careful analysis of each partially closed alcove in the Worpitzky partition (Theorem \ref{thm:Worpitzky}). 
Let us proceed the proof with $\Psi$ as general as possible to see how the compatibility is crucial. Since both sides are quasi-polynomials, it suffices to prove the formula for $q\ge\h$. 
In which case, $\chi^{\quasi}_{\Psi}(q) \ge \chi^{\quasi}_{\Phi^+}(q) >0$ (Corollary \ref{cor:>0}), and $q-\h+\dsc_\Psi(A_i^\circ)\ge0$ (Lemma \ref{lem:range}) for all $i \in [N]$ whence $(E_\Psi(S){\rm L}_{\overline{A^\circ}})(q)>0$.
Fix $i \in [N]$. 
The intersection of $A_i^\diamondsuit$ and the affine hyperplanes w.r.t. roots in $\Psi$ consists of $a$ facets $(0 \le a \le \ell+1)$, say $K_1,\ldots, K_a$, and $b$ faces of dimension $\le \ell-2$ $(b\ge 0)$, say $G_1,\ldots, G_b$ of $A_i^\diamondsuit$. 
Set $A:= \cup_{s=1}^a K_s$ and $B:=\cup_{j=1}^b G_j$.
The following elementary identity
$$A_i^\diamondsuit \setminus (A \cup B) = (A_i^\diamondsuit  \setminus A) \setminus (B \setminus A)$$
implies that
$$qA_i^\diamondsuit  \cap Z(\Phi^\vee) \setminus  \bigcup_{\mu\in \Psi,k \in \Z}H_{\mu,kq} = 
\left( q(A_i^\diamondsuit \setminus  A) \cap Z(\Phi^\vee)  \right)  
\setminus \left( q(B \setminus A) \cap Z(\Phi^\vee) \right).
$$
Denote $\scR(A_i^\diamondsuit; q) :=\# \left( q(B \setminus A) \cap Z(\Phi^\vee) \right)$.
In other words, $\scR(A_i^\diamondsuit; q)$ is the number of coweight lattice points in (the union of) the non-facets $qG_j$ but not in any facet $qK_s$ of $qA_i^\diamondsuit$. 
By Definition \ref{def:partial}, we may write
\begin{equation}
\label{eq:both}  
 \#\left(qA_i^\diamondsuit  \cap Z(\Phi^\vee) \setminus  \bigcup_{\mu\in \Psi,k \in \Z}H_{\mu,kq}  \right) =  \#  \left\{x \in Z(\Phi^\vee) \middle|
\begin{array}{c}
      (\alpha,x) > qk_\alpha  \,(\alpha \in I )\\
       (\beta,x) < qk_\beta  \,(\beta \in J\cap\Psi)\\
        (\delta,x) \le qk_\delta  \,(\delta \in J\cap\Psi^c)
    \end{array}
\right\} - \scR(A_i^\diamondsuit; q),
\end{equation}
where each facet $qK_s$ is supported by one of the hyperplanes of the form $H_{\beta,qk_\beta}$.

Since $A_i^\circ=w(A^\circ)+\gamma$ for some $w \in W$ and $\gamma \in Z(\Phi^\vee)$ ($\gamma$ is uniquely is determined by $w$, e.g.,  \cite[Theorem 4.9]{H90}), the dilation $qA_i^\circ$ can be written as
\begin{equation*}\label{eq:qA-explain}
qA_i^\circ= \left\{x \in V \middle|
\begin{array}{c}
      (-w(\alpha_0),x) < q((-w(\alpha_0),\gamma)+1),  \\
      (-w(\alpha_i),x) < q(-w(\alpha_i),\gamma),\, (1\le i \le \ell)
    \end{array}
\right\}.
 \end{equation*}
Thus the half-spaces defined by $(\delta,x) \le qk_\delta \,(\delta \in J\cap\Psi^c) $ correspond exactly to the roots $\alpha_i \in \widetilde{\Delta}$ satisfying $-w(\alpha_i) \in \Psi^c$. 
Applying Proposition \ref{prop:remove} and Definition \ref{def:asc-ext}, we obtain
\begin{equation}
\label{eq:remove-facets}  
\#  \left\{x \in Z(\Phi^\vee) \middle|
\begin{array}{c}
      (\alpha,x) > qk_\alpha  \,(\alpha \in I )\\
       (\beta,x) < qk_\beta  \,(\beta \in J\cap\Psi)\\
        (\delta,x) \le qk_\delta  \,(\delta \in J\cap\Psi^c)
    \end{array}
\right\} 
=
 {\rm L}_{\overline{A^\circ}}(q-\h+\dsc_\Psi(A_i^\circ)).
\end{equation}

Moreover, by the Worpitzky partition and Proposition \ref{prop:bijection}, we have
\begin{equation}
\label{eq:applyWorpitzky} 
\chi^{\quasi}_{\Psi}(q)  = \sum_{i \in [N]} \#\left(qA_i^\diamondsuit  \cap Z(\Phi^\vee) \setminus  \bigcup_{\mu\in \Psi,k \in \Z}H_{\mu,kq}  \right).
\end{equation}
Using Theorem \ref{thm:equivalent} together with the shift operator (Definition \ref{def:shift-operator}), we have
\begin{equation}
\label{eq:apply-shift} 
(E_\Psi(S){\rm L}_{\overline{A^\circ}})(q)=\sum_{i \in [N]} {\rm L}_{\overline{A^\circ}}(q-\h+\dsc_\Psi(A_i^\circ)).
\end{equation}

Summarizing Formulas \eqref{eq:both}-\eqref{eq:apply-shift}, it follows that for every $\Psi \subseteq \Phi^+$ we have
$$\chi^{\quasi}_{\Psi}(q)= (E_\Psi(S){\rm L}_{\overline{A^\circ}})(q) - \sum_{i \in [N]} \scR(A_i^\diamondsuit; q).$$
The compatibility of $\Psi$ forces $\scR(A_i^\diamondsuit; q)=0$ for all $i \in [N]$.
This completes the proof of ``$\Rightarrow$".

Now suppose that $\Psi$ is not compatible. Then there exist $A_i^\diamondsuit$ and its non-facet $G_j$ such that $G_j \nsubseteq K_s$ for all facets $K_s$.
To prove ``$\Leftarrow$", it suffices to show that there exists (sufficiently large) $q \in \Z_{>0}$ such that the relative interior of $q(G_j \setminus A)$ contains a coweight lattice point, in which case, $\scR(A_i^\diamondsuit; q)>0$. 
Since $W_{\rm aff}$ preserves the lattice points of $qG_j$, we can transform the problem to (the relative interior of) a face of $q\overline{A^\circ}$. 
Since a face of $q\overline{A^\circ}$ is the convex hull of a nonempty subset of $\{0,q\varpi^\vee_1/c_1 ,\ldots,q\varpi^\vee_\ell/c_\ell\}$, the relative interior of that face contains a lattice point for sufficiently large $q$ (e.g., the center of mass of any face is a lattice point when $q = \lcm(1,\ldots, \ell+1)\cdot\lcm(c_1,\ldots, c_\ell)$).
The conclusion follows.
 \end{proof}
 
  \begin{remark}
\label{rem:non-com}
 \begin{enumerate}[(a)]
\item 
When $\Psi$ is not compatible, although $\chi^{\quasi}_{\Psi}(q)$ and  $(E_\Psi(S){\rm L}_{\overline{A^\circ}})(q)$ are never equal as quasi-polynomials, the equality may occur for particular constituents (see Example \ref{ex:yes-no}(b)).
Proof of Theorem \ref{thm:shift} hints that in order to find a formula of the characteristic quasi-polynomial of a non-compatible subset, we need to count the lattices points in $d$-simplices with $d \le \ell$. 
\item When the rank of the root system is at most $3$, it is possible to examine the (non-)compatibility by using the pictures. 
The drawing is rather difficult in higher-dimensional cases, in this regard, Theorem \ref{thm:shift}  is useful to check the (non-)compatibility (see Example \ref{ex:non-comp}).
\end{enumerate}
\end{remark}
  
  \begin{definition}
\label{def:ideal}
A subset $\Psi\subseteq\Phi^+$ is an \emph{ideal} of $\Phi^+$ (or a lower set of the root poset $( \Phi^+, \ge)$) if for $\beta_1,\beta_2 \in \Phi^+$, $\beta_1 \ge \beta_2, \beta_ 1 \in \Psi$ implies $\beta_2 \in \Psi$.
\end{definition}

Our second main result in the paper is: every ideal is compatible.
We remark that if $\ell \ge 2$ and $\delta \in \Phi^+$, then the compatible subset $\Phi^+ \setminus \{\delta\}$ is not an ideal unless $\delta$ is the highest root. 
We recall some basic facts on root systems.

For any $M \subseteq \Phi$, set $\Phi_M(\F) := \Phi \cap \F M$ where $\F= \Z, \Q$ or $\R$. 
Note that $\Phi_M(\Z) \subseteq \Phi_M(\Q)=\Phi_M(\R)$ and these root subsystems (in the real vector space that they span) have the same rank. 
In particular, if $M \subseteq \Delta$, then these three subsystems are identical.
 \begin{lemma}
\label{lem:sommers}
Let $\scJ \subsetneq \widetilde{\Delta}$. 
Then ${\scJ}$ is a base for $\Phi_{\scJ}(\Z)$.
 \end{lemma} 
\begin{proof} 
See, e.g., \cite[Lemma 4.3]{Sommers97}. 
\end{proof}

 \begin{lemma}
\label{lem:w-sommers}
For $w\in W$ and ${\scJ} \subsetneq \widetilde{\Delta}$, define $D := -w({\scJ}) = \{ -w(\alpha) \mid \alpha \in {\scJ}\} \subseteq \Phi$. 
Then $D$ is a base for $\Phi_{D}(\Z)$.
 \end{lemma} 
\begin{proof} 
$D$ is linearly independent because $(-w(\alpha_j), -w(\alpha_k)) = (\alpha_j, \alpha_k)$ for all $\alpha_j, \alpha_k \in {\scJ}$. 
 Let $\gamma \in  \Phi_{D}(\Z)$.
 We can write $\gamma  = \sum_{\delta \in D} d_\delta\delta$ where $d_\delta \in \Z$. 
 We need to show that all $d_\delta \ge0$ or all $d_\delta \le0$. 
 Note that $\scJ$ and $D$ are isomorphic via $w$.
 Rewriting, we have $\gamma  = -w(\sum_{\alpha \in {\scJ}} d_{-w(\alpha)}\alpha)$. 
Thus, $\sum_{\alpha \in {\scJ}} d_{-w(\alpha)}\alpha \in   \Phi_{{\scJ}}(\Z)$.
Lemma \ref{lem:sommers} completes the proof.
 \end{proof}

 \begin{theorem}
\label{thm:ideal-com}
 If $\Psi\subseteq\Phi^+$ is an ideal, then $\Psi$ is compatible.
 \end{theorem}

\begin{proof}  
 If $A_i^\diamondsuit \cap H_{\alpha, m_\alpha}$ is a non-empty face but not a facet of $A_i^\diamondsuit$ for $i \in [N]$, $\alpha\in \Psi,m_\alpha \in \Z$, then we can write $A_i^\diamondsuit \cap H_{\alpha, m_\alpha}= \cap_{j=1}^n H_{\beta_j, m_{\beta_j}}\cap  A_i^\diamondsuit$, where $H_{\beta_j, m_{\beta_j}}$ ($\beta_j \in \Phi^+\setminus \{\alpha\}$) are the supporting hyperplanes of the facets of $A_i^\diamondsuit$ for all $1 \le j \le n$ with $2 \le n \le \ell$. 
 Set $D:=\{\beta_j \mid 1 \le j \le n\}$. 
  We need to show that there is $\beta_k \in \Psi \cap D$. 

 We have the following facts.
\begin{enumerate}
\item[Fact 1.]   $D$ is a base for $\Phi_{D}(\Z)$. 
Let $w\in W$ so that $A_i^\circ=w(A^\circ)+\gamma_w$ ($\gamma_w \in Z(\Phi^\vee)$).
 It follows from the definition of $A_i^\diamondsuit$ that there exists ${\scJ} \subsetneq \widetilde{\Delta}$, such that $D = -w({\scJ})$. 
 The rest follows from Lemma \ref{lem:w-sommers}.

\item[Fact 2.]  $\alpha \in\Phi_{D}(\Q)$. Choose $z_0 \in A_i^\diamondsuit \cap H_{\alpha, m_\alpha}$, the translation $t_{z_o}: V \to V$ via $t_{z_o}(x) = -z_0+x$ yields $-z_0 + {\rm{affine hull}}(A_i^\diamondsuit \cap H_{\alpha, m_\alpha})= \cap_{j=1}^n H_{\beta_j,0} \subseteq H_{\alpha, 0}$. 
Thus, $\alpha \in \Phi_D(\R)=\Phi_D(\Q)$.
\end{enumerate}
\textbf{Claim.} $\alpha \in\Phi_{D}(\Z)$. 
By Fact 2, $D \cup\{\alpha\}$ is linearly dependent. 
Then there exist integers $d_\alpha, d_{\beta_1}, \ldots, d_{\beta_n}\in\Z$ not all zero with $d_\alpha\ne0$, $\gcd\{|d_\alpha|,| d_{\beta_1}|, \ldots, |d_{\beta_n}|\}=1$ such that $d_\alpha\alpha + \sum_{j=1}^n d_{\beta_j}\beta_j=0$ (see, e.g., \cite[Lemma 17]{BV07}). 
The claim is proved once we prove $d_\alpha \in \{-1,1\}$. 

\noindent
Let us transform the situation in $A_i^\diamondsuit$ to {\color{blue!60}{$\overline{A^\circ}$}}.
Set  $\alpha := -w(\alpha')$ with $\alpha'\in\Phi$. 
We can rewrite the relation above as 
\begin{equation}
\label{eq:relation} 
d_\alpha\alpha' + \sum_{\alpha_j \in \scJ} d_{\alpha_j}\alpha_j=0,
\end{equation}
where $d_{\alpha_j} := d_{\beta}$ with $\beta \in D$ such that $\beta = -w(\alpha_j) $.
We consider two cases.
\begin{enumerate}
\item[Case 1.]   
$\scJ \subsetneq \Delta$. 
Since $\alpha' \in \Z\Delta$, the linearly independence of $\Delta$ and Relation \eqref{eq:relation}   yield $\alpha' \in \Z \scJ$. 
Additionally, $d_\alpha$ divides $d_{\alpha_j}$ for all $\alpha_j \in \scJ$. 
Thus, $d_\alpha \in \{-1,1\}$. 

\item[Case 2.]  
${\scJ}={\I}\cup\{\alpha_0\}$ with $\emptyset \ne {\I} \subsetneq \Delta$. 
If $d_{\alpha_0}=0$, then the same argument as in Case 1 yields $d_\alpha \in \{-1,1\}$. 
Assume that $d_{\alpha_0}\ne0$ and $\alpha'\in\Phi^+$. 
The transformation by $w$ implies that
\begin{equation*}
\label{eq:relation-intersect} 
\overline{A^\circ} \cap H_{\alpha', 1}= \cap_{\alpha_j \in {\I}} H_{\alpha_j, 0}\cap H_{\alpha_0, -1}\cap  \overline{A^\circ} \subseteq \overline{P^\diamondsuit}.
\end{equation*}
\end{enumerate}
Let $x = \sum_{\alpha_s \in\Delta} r_{\alpha_s} \varpi^\vee_{\alpha_s} \in \overline{P^\diamondsuit}$ ($r_{\alpha_s} \in [0,1]$) be a typical element in both sets. 
Since $x\in H_{\alpha_j, 0}$ for all $\alpha_j \in {\I}$, we can write $x=  \sum_{\alpha_s \in\Delta \setminus \I} r_{\alpha_s} \varpi^\vee_{\alpha_s}$. 
Thus, 
$$(x, -\sum_{\alpha_j \in \scJ} d_{\alpha_j}\alpha_j)= (x, -d_{\alpha_0}\alpha_0)=d_{\alpha_0}.$$
It follows from  Relation \eqref{eq:relation} and the above equation  that
$$d_\alpha = (x, d_\alpha\alpha') = (x, -\sum_{\alpha_j \in \scJ} d_{\alpha_j}\alpha_j)=d_{\alpha_0}.$$
Thus we can rewrite Relation \eqref{eq:relation}   as
$$
d_\alpha(\alpha' +\alpha_0)+ \sum_{\alpha_j \in {\I}} d_{\alpha_j}\alpha_j=0.
$$
Note that $\alpha' +\alpha_0 =-w^{-1}(\alpha+\beta') \ne0$ where $\beta':=-w(\alpha_0) \in D$.
A similar argument as in Case 1 yields $d_\alpha \in \{-1,1\}$. 
Repeat for $-\alpha'$ if $\alpha'\in -\Phi^+$. 
 
By Fact 1 and \textbf{Claim}, we can write $\alpha = \sum_{j=1}^np_j\beta_j$ with at least two $p_{k_1}\ge 1$, $p_{k_2}\ge1$. 
Thus, there is $\beta_k \in D$ such that $\alpha \ge_{\Phi^+} \beta_k$. 
Since $\Psi$ is an ideal, we must have $\beta_k \in \Psi$. 
This completes the proof.
 \end{proof}  

  \begin{remark}
\label{rem:ideal-compare}
There exists a compatible subset $\Psi$ that is not an ideal w.r.t. any positive system of $\Phi$ (see Example \ref{ex:yes-no}(d)).
The characteristic quasi-polynomial  $\chi_{\Psi}^{\mathrm{quasi}}(q)$ when $\Psi$ is an ideal and $\Phi$  is of classical type has been computed by using information of the signed graph associated with $\Psi$  \cite{T19}. 
It would be interesting to compare the mentioned computation with Theorem \ref{thm:shift}.
\end{remark}

 \begin{example}
\label{ex:yes-no}
Let $\Phi=G_2$ as in Figure \ref{fig:G2}. 
 \begin{enumerate}[(a)]
\item 
Let $\Psi=\Phi^+ \setminus  \{ \tilde{\alpha}\}$. 
We can see that $\Psi$ is compatible either by Example \ref{ex:compatibleW}(a) or by Theorem \ref{thm:ideal-com}. By definition of  characteristic quasi-polynomial (Theorem \ref{thm:KTT}),
\begin{align*}
\chi^{\quasi}_{\Psi}(q) & = \#\left\{ \textbf{z} \in  \Z_q^2 \mid z_1,z_2,z_1+z_2, 2z_1+z_2, 3z_1+z_2\ne \overline{0 } \right\} \\
& = \begin{cases}
(q-1)(q-4) & \mbox{if $q \equiv 1,5 \bmod 6$}, \\
(q-2)(q-3) & \mbox{if $q \equiv 2,3,4 \bmod 6$}, \\
q^2-5q+8  & \mbox{if $q \equiv 0 \bmod 6$}.
\end{cases}
\end{align*}
Note that the Ehrhart quasi-polynomial of the fundamental alcove (w.r.t. the coweight lattice) of every root system has been completely computed, e.g., in \cite{Su98}. 
$$
{\rm L}_{\overline{A^\circ}}(q)
= \begin{cases}
\frac{1}{12}(q+1)(q+5) & \mbox{if $q \equiv 1,5 \bmod 6$}, \\
\frac{1}{12}(q+2)(q+4) & \mbox{if $q \equiv 2,4 \bmod 6$}, \\
\frac{1}{12}(q+3)^2 & \mbox{if $q \equiv 3 \bmod 6$}, \\
\frac{1}{12}(q^2+6q+12)  & \mbox{if $q \equiv 0 \bmod 6$}.
\end{cases}
$$
Moreover, from Table \ref{tab:compute}, we have $E_\Psi(t)=  8t^6 + 2t^5+2t^4$. Thus,
$$
(E_\Psi(S){\rm L}_{\overline{A^\circ}})(q) = 8{\rm L}_{\overline{A^\circ}}(q-6)+2{\rm L}_{\overline{A^\circ}}(q-5)+2{\rm L}_{\overline{A^\circ}}(q-4),
$$
which coincides with the computation of $\chi^{\quasi}_{\Psi}(q)$ above, and justifies Theorem \ref{thm:shift}.
\begin{table}[htbp]
\centering
{\renewcommand\arraystretch{1.5} 
\begin{tabular}{|c|c|}
\hline
$\dsc_\Psi(w)$ & $w$  \\
\hline\hline
$2$ &  $s_2s_1s_2, (s_1s_2)^2$  \\
\hline
$1$ & $1, s_1$ \\
\hline
$0$ &  $s_2,s_1s_2, s_2s_1, s_1s_2s_1, (s_2s_1)^2, s_1(s_2s_1)^2, s_2(s_1s_2)^2, (s_2s_1)^3$  \\
\hline
\end{tabular}
}
\bigskip
\caption{Computation of $\dsc_\Psi(W)$ when $\Psi=\Phi^+(G_2) \setminus  \{ 3\alpha_1+2\alpha_2\}$. Here $s_1,s_2$ are the reflections w.r.t. $\alpha_1, \alpha_2$, respectively.}
\label{tab:compute}
\end{table}
 \item 
Let $\Psi' = \{\alpha_1, \alpha_1+\alpha_2, 3\alpha_1+2\alpha_2\}$. Then $\Psi'$ is not compatible by Example \ref{ex:compatibleW}(b). We may compute
\begin{align*}
\chi^{\quasi}_{\Psi'}(q) & = \begin{cases}
(q-1)(q-2) & \mbox{if $q \equiv 1 \bmod 2$}, \\
q^2-3q+3 & \mbox{if $q \equiv 0 \bmod 2$}.
\end{cases}
\\
E_{\Psi'}(t) & = 5t^6 + t^5+2t^4+ 3t^3 + t^2,
\\
(E_{\Psi'}(S){\rm L}_{\overline{A^\circ}})(q) & = 
\begin{cases}
(q-1)(q-2) & \mbox{if $q \equiv 1,5 \bmod 6$}, \\
q^2-3q+3  & \mbox{if $q \equiv 2,4 \bmod 6$}, \\
q^2-3q+4 & \mbox{if $q \equiv 3 \bmod 6$}, \\
q^2-3q+5  & \mbox{if $q \equiv 0 \bmod 6$}.
\end{cases}
\end{align*}
Thus $\chi^{\quasi}_{\Psi'}(q)  = (E_{\Psi'}(S){\rm L}_{\overline{A^\circ}})(q)$ if $q \equiv 1,2,4, 5 \bmod 6$, and $\chi^{\quasi}_{\Psi'}(q)  < (E_{\Psi'}(S){\rm L}_{\overline{A^\circ}})(q)$ otherwise. 
In particular, we can easily check on Figure \ref{fig:q=9} that $\chi^{\quasi}_{\Psi'}(7)  = (E_{\Psi'}(S){\rm L}_{\overline{A^\circ}})(7)=30$. 
 \item 
 In addition to the example above, both sides of the formula in Theorem \ref{thm:shift} have different values at each constituent when $\Psi$ is not compatible.
For example, let $\Psi_1 = \{ \alpha_1+\alpha_2\}$ as one of the non-compatible singleton subsets mentioned in Example \ref{ex:compatibleW}(b).
Then
\begin{align*}
\chi^{\quasi}_{\Psi_1}(q) & = q(q-1) \mbox{ for all } q, \\
E_{\Psi_1}(t) & = t^6 + 2t^5+3t^4+ 3t^3 + 2t^2+t,
\\
(E_{\Psi_1}(S){\rm L}_{\overline{A^\circ}})(q) & = q(q-1) +1\mbox{ for all } q.
\end{align*}
Thus $\chi^{\quasi}_{\Psi_1}(q) <  (E_{\Psi_1}(S){\rm L}_{\overline{A^\circ}})(q)$ for all $q$.
\item 
Let $\Psi_2 = \{\alpha_2, 3\alpha_1+\alpha_2\}$. Then $\Psi_2$ is compatible (e.g., checked by Figure \ref{fig:q=9}) but not an ideal of any positive system of $\Phi=G_2$. Otherwise, $\Psi_2$ must be the associated base w.r.t. that positive system since it has only two elements. However, it is a contradiction because the length of $\alpha_2$ and $3\alpha_1+\alpha_2$ are equal.
\end{enumerate}
\end{example}
 \begin{example}
\label{ex:non-comp}
Let $\Phi=A_4$ and let $\Psi=\Delta \cup  \{  \tilde{\alpha}\}= \{\alpha_1, \alpha_2, \alpha_3, \alpha_4,\alpha_1+ \alpha_2+ \alpha_3+ \alpha_4 \}$. We may compute 
\begin{align*}
\chi^{\quasi}_{\Psi}(q) & =q^4 -5q^3 +10q^2 -10q +4  \mbox{ for all } q, \\
E_{\Psi}(t) & = 5t^5+9t^4+ 9t^3 + t^2,
\\
(E_{\Psi}(S){\rm L}_{\overline{A^\circ}})(q) & = q^4 -5q^3 +11q^2 -12q +5 \mbox{ for all } q.
\end{align*}
Thus, $\chi^{\quasi}_{\Psi}(q) <  (E_{\Psi}(S){\rm L}_{\overline{A^\circ}})(q)$ for all $q \ne 1$.
It follows from Theorem \ref{thm:shift}  that $\Psi$ is not compatible.

\end{example}

In general, it is not easy to compute the $\A$-Eulerian polynomial of an arbitrary (compatible) subset by using its primary definition (Definition \ref{def:Eulerian}).
For some special cases, e.g., the subsets obtained from the positive system by removing one element as mentioned in Example \ref{ex:compatibleW}(a), we can compute the corresponding Eulerian polynomial in terms of the root system invariants.

Let $\Phi_l$ (resp., $\Phi_s$) denote the set of all long (resp., short) roots of $\Phi$. 
We will use the notation $\Phi_{l,s}$ when we refer to either $\Phi_l$ or   $\Phi_s$.
If $\Phi$ is simply-laced, we agree that $\Phi_l=\Phi, \Phi_s=\emptyset$. 
Denote $\Phi_{l,s}^+:=\Phi^+ \cap \Phi_{l,s}$
and $\widetilde{\Delta}_{l,s}:=\widetilde{\Delta}\cap \Phi_{l,s}$.

\begin{proposition}
\label{prop:many}
Let $\delta \in \Phi_{l,s}^+$ and set $\delta^c : = \Phi^+ \setminus \{\delta\}$. Define
$$\Omega:=\{ w \in W \mid \delta = -w(\alpha_i) \mbox{ for some } \alpha_i \in \widetilde{\Delta}_{l,s}\}.$$
Then 
\begin{enumerate}[(i)]
\item $\Omega = \{ w \in W \mid \dsc_{\delta^c}(w) = c_i \mbox{ for some } \alpha_i \in \widetilde{\Delta}_{l,s}\} =  \{ w \in W \mid \dsc_{\delta^c}(w)>0\}$.
\item $\Omega = \bigsqcup_{\alpha_i \in \widetilde{\Delta}_{l,s}}\Omega_i$, where  $\Omega_i:=\{ w \in W \mid \delta = -w(\alpha_i)\}$ for each $ \alpha_i \in \widetilde{\Delta}_{l,s}$.
\item $\Omega_i\simeq  {\rm stab}(\delta)$ (as sets), where ${\rm stab}(\delta):=\{ w \in W \mid \delta = w(\delta)\}$ is the stablizer of $\delta$ in $W$. In particular, $\#\Omega_i=   \frac{\#W}{\#\Phi_{l,s}}$. 
\item $\#\Omega = \frac{\#\widetilde{\Delta}_{l,s}\times\#W}{\#\Phi_{l,s}}$.
\end{enumerate}
 \end{proposition}
\begin{proof}  
(i) and (ii) are trivial, and (iv) is an immediate consequence of (ii) and (iii). 
Let us prove (iii).
Fix $w_{\delta \rightarrow \alpha_i} \in W$ so that $w_{\delta \rightarrow \alpha_i} (\delta)=\alpha_i$.
Thus, if $w \in \Omega_i$, then $ws_{\alpha_i}w_{\delta \rightarrow \alpha_i} \in  {\rm stab}(\delta)$.
Consider the set map
$$\phi: \Omega_i \to  {\rm stab}(\delta), \mbox{ defined by } w \mapsto ws_{\alpha_i}w_{\delta \rightarrow \alpha_i}.$$
Clearly, $\phi$ is well-defined and injective. 
Moreover, for any $\tau \in {\rm stab}(\delta)$, we have $w=\tau w^{-1}_{\delta \rightarrow \alpha_i}s_{\alpha_i}\in \Omega_i$ such that $\phi(w)=\tau$.
Therefore $\phi$ is surjective and hence bijective. 
To prove the second statement, it suffices to show that  $\#{\rm stab}(\delta)= \frac{\#W}{\#\Phi_{l,s}}$. 
This can be done by proving that the surjective map $\varphi: W\to \Phi_{l,s}$
defined by $w \mapsto w(\delta)$ induces a bijection $W/{\rm stab}(\delta) \to \Phi_{l,s}$.
 \end{proof}

\begin{theorem}\label{thm:niceA'}   
If $\delta \in \Phi_{l,s}^+$, then 
\begin{equation*}\label{eq:niceA'}
E_{\delta^c}(t) =\frac{\#W}{f\times\#\Phi_{l,s}} \sum_{\alpha_i \in \widetilde{\Delta}_{l,s}}  t^{\h-c_i}+\frac{\#W(\#\Phi_{l,s}-\# \widetilde{\Delta}_{l,s})}{f\times\#\Phi_{l,s}}t^{\h}.
\end{equation*}
\end{theorem}
\begin{proof} 
It follows directly from Definition \ref{def:Eulerian} and Proposition \ref{prop:many}.
\end{proof}

An example of Theorem \ref{thm:niceA'}  is already mentioned in Example \ref{ex:yes-no}(a). 
In addition, if $\delta$ is any short root in $\Phi^+(G_2)$, then Theorem \ref{thm:niceA'} implies that 
$E_{\delta^c}(t) = 10t^6+2t^3$.

 Next, we discuss the generating function of the characteristic quasi-polynomial.
\begin{proposition}
\label{prop:ehrhart}   
 $$\sum_{q \ge 1}{\rm L}_{A^\circ}(q)t^q = \frac{t^\h}{\prod_{i=0}^\ell (1-t^{c_i}) }.$$
 \end{proposition}
\begin{proof} 
See, e.g., \cite[Proof of Theorem 3.1]{KTT10}.
\end{proof}

\begin{theorem}\label{thm:known}  
\begin{enumerate}[(i)]
\item $$\sum_{q \ge 1}\chi^{\quasi}_{\Phi^+}(q)t^q = \frac{ \frac{\#W}{f}t^\h}{\prod_{i=0}^\ell (1-t^{c_i}) }.$$ 
\item  Let $R_\Phi(t)$ be the generalized Eulerian polynomial (see Remark \ref{rem:empty-full1}(b)).
Then
 $$\sum_{q \ge 1}q^\ell t^q = \frac{R_\Phi(t)}{\prod_{i=0}^\ell (1-t^{c_i}) }.$$ 
\end{enumerate}
\end{theorem} 
\begin{proof} 
For a proof of (i), see, e.g.,  \cite{A96}, \cite[Theorem 4.1]{BS98}, \cite[Theorem 3.1]{KTT10}.
(ii) follows from \cite[Theorem 10.1]{LP18}.
\end{proof}
 
\begin{theorem}\label{thm:Eulerian}   
A subset $\Psi \subseteq \Phi^+$ is compatible if and only if  
\begin{equation}\label{eq:gen-func}
 \sum_{q \ge 1}\chi^{\quasi}_{\Psi}(q)t^q = \frac{E_\Psi(t)}{\prod_{i=0}^\ell (1-t^{c_i}) }.
\end{equation}
 \end{theorem} 
 
\begin{proof} 
Using Corollary \ref{cor:>0}, Propositions \ref{prop:ehrhart} and \ref{prop:remove}, we compute 
\begin{align*}
\sum_{q \ge 1}(E_\Psi(S){\rm L}_{\overline{A^\circ}})(q)t^q  & = \sum_{q \ge 1}\sum_{i \in [N]}   {\rm L}_{A^\circ}(q+\dsc_\Psi(A_i^\circ)) t^q  \\
&=\sum_{i \in [N]} t^{-\dsc_\Psi(A_i^\circ)} \sum_{q \ge 1}  {\rm L}_{{A^\circ}}(q+\dsc_\Psi(A_i^\circ)) t^{q+\dsc_\Psi(A_i^\circ)}   \\
&=\sum_{i \in [N]} t^{-\dsc_\Psi(A_i^\circ)} \sum_{n_i \ge \h}  {\rm L}_{{A^\circ}}(n_i) t^{n_i}   \\
&=\frac{\sum_{i \in [N]}t^{\h-\dsc_\Psi(A_i^\circ)}}{\prod_{i=0}^\ell (1-t^{c_i}) }  =  \frac{E_\Psi(t)}{\prod_{i=0}^\ell (1-t^{c_i}) }.
\end{align*}

If $\Psi \subseteq \Phi^+$ is compatible, then Formula \eqref{eq:gen-func} follows from Theorem \ref{thm:shift} and the calculation above.
Assume that Formula \eqref{eq:gen-func} holds but $\Psi$ is not compatible. By Proof of Theorem \ref{thm:shift}, we can write 
$$\chi^{\quasi}_{\Psi}(q)= (E_\Psi(S){\rm L}_{\overline{A^\circ}})(q) - R(q),$$
for some nonzero function $R(q)$ (actually quasi-polynomial) in $q$. Thus,
$\sum_{q \ge 1} R(q)t^q = 0$, which is a contradiction. 
 
\end{proof}

\begin{remark}
\label{rem:empty-full2}
By Remark \ref{rem:empty-full1}, Theorems \ref{thm:shift} and \ref{thm:Eulerian}, if
 \begin{enumerate}[(a)]
\item  
$\Psi = \Phi^+$, then we recover Theorems \ref{thm:quasi-Ehrhart} and \ref{thm:known}(i).
\item $\Psi =\emptyset$, then we recover \cite[Theorem 4.8]{Y18W} and Theorem \ref{thm:known}(ii).
\end{enumerate}
\end{remark}

 \begin{corollary}\label{cor:A'}   
If $\delta \in \Phi_{l,s}^+$, then 
\begin{equation*}
\label{eq:A'}
\begin{aligned}
& \chi^{\quasi}_{\delta^c}(q)  =\frac{\#W}{f\times\#\Phi_{l,s}} \sum_{\alpha_i \in \widetilde{\Delta}_{l,s}}  {\rm L}_{\overline{A^\circ}}(q-\h+c_i)+\frac{\#W(\#\Phi_{l,s}-\# \widetilde{\Delta}_{l,s})}{f\times\#\Phi_{l,s}} {\rm L}_{\overline{A^\circ}}(q-\h),  \\
& \sum_{q \ge 1}\chi^{\quasi}_{\delta^c}(q)t^q  = \frac{ \#W\sum_{\alpha_i \in \widetilde{\Delta}_{l,s}}  t^{\h-c_i}+\#W(\#\Phi_{l,s}-\# \widetilde{\Delta}_{l,s})t^{\h}}{f\times\#\Phi_{l,s}\times\prod_{i=0}^\ell (1-t^{c_i}) }.
\end{aligned}
\end{equation*}
\end{corollary}
\begin{proof} 
It follows  from Theorems \ref{thm:shift}, \ref{thm:niceA'} and \ref{thm:Eulerian}.
\end{proof}

\section{Deformations of Weyl subarrangements}
\label{sec:deform}
Let $\Psi \subseteq \Phi^+$, and recall the notation $\Psi^c=\Phi^+ \setminus \Psi$. 
Let $a\le b$ be integers, and denote $[a,b]:=\{m\in\Z\mid a \le m\le b\}$. 
Also, if $b \ge 1$, then write $[b]$ instead of $[1,b]$.
 \begin{definition}\label{def:types}
Let $a\le b$, $c \le d$ be integers. Define two types of the \emph{deformed Weyl arrangements} of $\Psi$ as follows:
\begin{enumerate}[(Type I)]
    \item $\mathrm{I}_\Psi^{[a,b]} := \{H_{\alpha,m} \mid \alpha \in \Psi, m \in [a,b]\}$.
    \item $\mathrm{II}_{\Psi,\Psi^c}^{[a,b],[c,d]} := \mathrm{I}_\Psi^{[a,b]}  \sqcup \mathrm{I}_{\Psi^c}^{[c,d]}$.
\end{enumerate}
\end{definition}

\begin{remark}
\label{rem:deform}
 \begin{enumerate}[(a)]
\item 
There is an obvious duality $\mathrm{II}_{\Psi,\Psi^c}^{[a,b],[c,d]} =\mathrm{II}_{\Psi^c,\Psi}^{[c,d],[a,b]}$. 
We can list some specializations: $\mathrm{I}_{\emptyset}^{[a,b]}=\varnothing$ the empty arrangement,  $\mathrm{II}_{\emptyset,\Phi^+}^{[a,b],[c,d]}= \A_{\Phi^+}^{[c,d]}$ the \emph{truncated affine Weyl arrangement}, including the extended Shi, Catalan, Linial arrangements, see, e.g., \cite[\S9]{SP00}. 
In addition, $\mathrm{I}_{\Phi^+}^{[a,b]} =\mathrm{II}_{\Psi,\Psi^c}^{[a,b],[a,b]}=\A_{\Phi^+}^{[a,b]}$. 
We refer the reader to \cite{A99}, \cite{A04}, \cite{Y18W} and \cite{Y18L} for more details on the characteristic (quasi-)polynomials of $\A_{\Phi^+}^{[a,b]}$.
\item 
The deformed Weyl arrangements of an arbitrary $\Psi$ are less well-known. When $\Phi$ is of type $A$, the \emph{deleted (or graphical) Shi arrangement}, see, e.g., \cite[\S3]{A96} or \cite{AR12}, is the product (e.g., \cite[Definition 2.13]{OT92}) of the $1$-dimensional  empty arrangenment and $\mathrm{II}_{\Psi,\Psi^c}^{[0,1],[0,0]}$. 
\end{enumerate}
\end{remark}

 \begin{definition}\label{def:others}   
For $w\in W$, define
\begin{align*}
\overline{\dsc}_\Psi(w) &:= \sum_{0 \le i \le \ell, \, w(\alpha_i)\in -\Psi}c_i, \\
\asc_\Psi(w) &:=  \sum_{0 \le i \le \ell, \, w(\alpha_i)\in \Psi^c}c_i, \\
\overline{\asc}_\Psi(w) &:= \sum_{0 \le i \le \ell, \, w(\alpha_i)\in \Psi}c_i.
\end{align*}
\end{definition}
Obviously, $\asc_\Psi(w)+\overline{\asc}_\Psi(w) +\dsc_\Psi(w)+\overline{\dsc}_\Psi(w)= \h$ for all $w\in W$. 
Similar to Lemma \ref{lem:range}, each function defined above takes values in $[0,\h-1]$.
Furthermore, 
\begin{align*}
\asc_\emptyset(w) &= \overline{\asc}_{\Phi^+}(w) = \asc_\Psi(w) + \overline{\asc}_\Psi(w), \\
\dsc_\emptyset(w) &= \overline{\dsc}_{\Phi^+}(w) = \dsc_\Psi(w) + \overline{\dsc}_\Psi(w).
\end{align*}
Similar to Definition \ref{def:asc-ext}, we can extend the functions above to functions on the set of all alcoves.

Now let us formulate a deformed version of Proposition \ref{prop:remove}. Set
 \begin{align*}
F_i^{[a_i,b_i]} & := \bigcup_{m \in [a_i,b_i]}H_{\alpha_i, m}\,\,(1\le i \le \ell), \\
F_0^{[a_0,b_0]} & := \bigcup_{m \in [a_0,b_0]}H_{\alpha_0,-q+ m}.
\end{align*}

\begin{proposition}
\label{prop:affine-remove}   
Let $\{i_1, \ldots, i_k\} \subseteq [0,\ell]$, and let $b_{i_j}\ge0$ for all $1 \le j \le k$.
Suppose that $q > \sum_{j\in[k]}(b_{i_j}+1)c_{i_j}$. Then
 $$\#\left(q \overline{A^\circ} \cap Z(\Phi^\vee)\setminus \bigcup_{j\in[k]} F_{i_j}^{[0, b_{i_j}]} \right)  = {\rm L}_{\overline{A^\circ}}(q- \sum_{j\in[k]}(b_{i_j}+1)c_{i_j}).$$
\end{proposition} 
\begin{proof} 
The formula was implicitly used in  \cite[\S5]{Y18W} and its proof is very similar to the non-deformed case. Note that if $i \in [\ell]$, then
 \begin{align*}
\#\left(q\overline{A^\circ} \cap Z(\Phi^\vee)\setminus F_{i}^{[0, b_i]} \right)
& = \#\left\{x \in Z(\Phi^\vee) \middle|
\begin{array}{c}
      (\alpha_i,x) \ge b_i+1  \\
       (\alpha_j,x) \ge 0 \,(j \in[\ell]\setminus\{ i\})\\
        (\alpha_0,x) \ge -q 
    \end{array}
\right\} \\
&=  {\rm L}_{\overline{A^\circ}}(q- (b_i+1)c_i).
\end{align*}
Here the last equality follows from the bijection $x \mapsto x+(b_i+1)\varpi_1^\vee$.
The proof for $i=0$ is similar. Then apply the formula above repeatedly. 
\end{proof}

\begin{remark}
\label{rem:intervals}
 If we replace the interval $[0, b_{i_j}]$ in Proposition \ref{prop:affine-remove} by $[a, b_{i_j}]$ with $a\ge1$, there might be a large change in the right-hand side of the formula above. For example, if $1\le a_1 \le b_1$, then
$$\#\left(q\overline{A^\circ} \cap Z(\Phi^\vee)\setminus F_{1}^{[a_1, b_1]} \right)=  {\rm L}_{\overline{A^\circ}}(q- (b_1+1)c_1)+{\rm L}_{\overline{A^\circ}}(q)- {\rm L}_{\overline{A^\circ}}(q-a_1c_1).$$
 \end{remark}

 \begin{theorem}\label{thm:CQP-I}
Let $\Psi$ be a compatible subset of $\Phi^+$.
 \begin{enumerate}[(i)]
\item If $a, b\ge 0$, then 
$$
\chi^{\quasi}_{\mathrm{I}_{\Psi}^{[-a,b]}}(q) = \sum_{i \in [N]} {\rm L}_{\overline{A^\circ}}(q-(b+1)\overline{\asc}_\Psi(A_i^\circ) -{\asc}_\Psi(A_i^\circ) -(a+1)\overline{\dsc}_\Psi(A_i^\circ)).
$$
\item  If $b\ge 1$, then 
$$
\chi^{\quasi}_{\mathrm{I}_{\Psi}^{[1,b]}}(q) = \sum_{i \in [N]} {\rm L}_{\overline{A^\circ}}(q-(b+1) \overline{\asc}_\Psi(A_i^\circ) - {\asc}_\Psi(A_i^\circ)).
$$
\end{enumerate}
\end{theorem}

\begin{proof} 
Proofs of (i) and (ii) are similar, and both are similar in spirit to the proof of \cite[Theorem 5.1]{Y18W}. 
See also Proof of Theorem \ref{thm:shift} in this paper.
First, we give a proof for (i).
Since both sides are quasi-polynomials, it is sufficient to prove the equality for $q \gg 0$ (actually, $q>(a+b+3)\h$ is sufficient). 
By Proposition \ref{prop:affine-remove}, for $i \in [N]$, 
\begin{align*}
 &\,   \#\left(qA_i^\diamondsuit  \cap Z(\Phi^\vee) \setminus  \bigcup_{\mu\in \Psi,k\in \Z, m\in[-a, b]}H_{\mu,kq+m}  \right) \\
= &  \, \#  \left\{x \in Z(\Phi^\vee) \middle|
\begin{array}{c}
      (\alpha,x) \ge qk_\alpha +b+1\,(\alpha \in I \cap\Psi)\\
            (\eta,x) > qk_\eta\,(\eta \in I \cap\Psi^c)\\
       (\beta,x) \le qk_\beta -a-1\,(\beta \in J\cap\Psi)\\
        (\delta,x) \le qk_\delta \,(\delta \in J\cap\Psi^c)
    \end{array}
\right\}   \\
= &  \, {\rm L}_{\overline{A^\circ}}(q-(b+1)\overline{\asc}_\Psi(A_i^\circ) -{\asc}_\Psi(A_i^\circ) -(a+1)\overline{\dsc}_\Psi(A_i^\circ)).
\end{align*}

By Proposition \ref{prop:bijection}, we have
$$
\chi^{\quasi}_{\mathrm{I}_{\Psi}^{[-a,b]}}(q) = \sum_{i \in [N]} {\rm L}_{\overline{A^\circ}}(q-(b+1)\overline{\asc}_\Psi(A_i^\circ) -{\asc}_\Psi(A_i^\circ) -(a+1)\overline{\dsc}_\Psi(A_i^\circ)).
$$

For (ii), note that for $i \in [N]$, 
\begin{align*}
&\,  \#\left(qA_i^\diamondsuit  \cap Z(\Phi^\vee) \setminus  \bigcup_{\mu\in \Psi,k\in \Z, m\in[1, b]}H_{\mu,kq+m}  \right) \\
= &\, \#  \left\{x \in Z(\Phi^\vee) \middle|
\begin{array}{c}
      (\alpha,x) \ge qk_\alpha +b+1\,(\alpha \in I \cap\Psi)\\
            (\eta,x) > qk_\eta\,(\eta \in I \cap\Psi^c)\\
       (\beta,x) \le qk_\beta \,(\beta \in J\cap\Psi)\\
        (\delta,x) \le qk_\delta \,(\delta \in J\cap\Psi^c)
    \end{array}
\right\} \\
= &\, {\rm L}_{\overline{A^\circ}}(q-(b+1) \overline{\asc}_\Psi(A_i^\circ) - {\asc}_\Psi(A_i^\circ)).
\end{align*}
\end{proof}

\begin{remark}
\label{rem:type-I}
Theorem \ref{thm:CQP-I} is a generalization of several known results.
 \begin{enumerate}[(a)]
\item When $\Psi=\Phi^+$, we obtain \cite[Theorem 1.2]{A04} ($a=b \ge 0$), and Theorem 5.1 ($a=b-1 \ge 0$), Theorem 5.2 ($b \ge 1$), Theorem 5.3 ($b=n+k$, $a=k-1$, $n,k \ge 1$) in \cite{Y18W}. When $\Psi=\emptyset$, we obtain \cite[Theorem 4.8]{Y18W}.
\item When $a=b=0$ (and $\Psi \subseteq \Phi^+$ is compatible), we obtain Theorem \ref{thm:shift}.
 \end{enumerate}
\end{remark}

By using the same method, we have the following result for the arrangements of type II.
 \begin{theorem}\label{thm:CQP-II}
Let $\Psi$ be a compatible subset of $\Phi^+$.
 \begin{enumerate}[(i)]
\item If $a, b, c, d\ge 0$, then 
\begin{align*}
\chi^{\quasi}_{\mathrm{II}_{\Psi,\Psi^c}^{[a,b],[c,d]}}(q) &= \sum_{i \in [N]} {\rm L}_{\overline{A^\circ}}(q-(b+1) \overline{\asc}_\Psi(A_i^\circ) -(d+1){\asc}_\Psi(A_i^\circ) \\
  & \qquad \qquad \quad - (a+1) \overline{\dsc}_\Psi(A_i^\circ) -(c+1){\dsc}_\Psi(A_i^\circ)).
\end{align*}
\item  If $a,b \ge 0, d \ge 1$, then 
\begin{align*}
\chi^{\quasi}_{\mathrm{II}_{\Psi,\Psi^c}^{[a,b],[1,d]}}(q)  &= \sum_{i \in [N]} {\rm L}_{\overline{A^\circ}}(q-(b+1) \overline{\asc}_\Psi(A_i^\circ) - (d+1){\asc}_\Psi(A_i^\circ) \\
  & \qquad \qquad \quad -(a+1) \overline{\dsc}_\Psi(A_i^\circ)).
\end{align*}
\item  If $b, d \ge 1$, then 
$$
\chi^{\quasi}_{\mathrm{II}_{\Psi,\Psi^c}^{[1,b],[1,d]}}(q)  = \sum_{i \in [N]} {\rm L}_{\overline{A^\circ}}(q-(b+1) \overline{\asc}_\Psi(A_i^\circ) - (d+1){\asc}_\Psi(A_i^\circ)).
$$
\end{enumerate}
\end{theorem}

\begin{remark}
\label{rem:more-components} \begin{enumerate}[(a)]
\item One can work with other intervals $[a,b]$ but the computation may become more complicated (see Remark \ref{rem:intervals}).
\item  One can define and study the arrangement $\bigsqcup_{k=1}^n \mathrm{I}_{\Psi_k}^{[a_k,b_k]}$ where $\Phi^+=\bigsqcup_{k=1}^n \Psi_k$ with $n \ge 3$. 
See, e.g.,  \cite[Theorem 3.11]{A96} for an example when $n=3$.
We choose not to develop this direction here.
 \end{enumerate}
\end{remark}

It would be interesting to  characterize the compatibility in the case of type $A$ (e.g., in terms of graphs) and compare the following result with \cite[Theorem 3.2]{AR12} and \cite[Theorem 3.9]{A96} (see Remark \ref{rem:deform} for the notation).
\begin{corollary}\label{cor:compare}   
Define $M_\Psi(t):= \frac1f \sum_{w\in W}t^{\h+\overline{\asc}_\Psi(w)}$. 
If $\Psi$ is compatible, then
$$\chi^{\quasi}_{\mathrm{II}_{\Psi,\Psi^c}^{[0,1],[0,0]}}(q)  = (M_\Psi(S){\rm L}_{\overline{A^\circ}})(q).$$
\end{corollary}

\noindent
\textbf{Acknowledgements.} 
The first author was partially supported by Mitacs Canada Globalink Research Award to visit Japan and carry out the collaboration.
The second author  is partially supported by JSPS Research Fellowship for Young Scientists Grant Number 19J12024.
The third author  is partially supported by JSPS KAKENHI Grant Number JP18H01115. 
The second author would like to thank Akiyoshi Tsuchiya for pointing out an 
error in the proof of the main result of the manuscript (Theorem \ref{thm:shift}) in a previous version.

\bibliographystyle{alpha} 
\bibliography{references}

\end{document}